\documentclass{article}

\usepackage{amsmath,amsfonts,amssymb}

\usepackage{graphicx}

\usepackage{dsfont}
\usepackage{bbold}
\usepackage{amsthm}
\usepackage[utf8]{inputenc}
\usepackage{geometry}

\newtheorem{thm}{Theorem}

\newtheorem{defn}[thm]{Definition}
\newtheorem{prp}{Proposition}
\newtheorem*{claim}{Claim}
\newtheorem{lemma}{Lemma}

\title{\textbf{Kakeya-type sets for Geometric Maximal Operators}}

\author{Anthony Gauvan\footnote{Institut Mathématiques d'Orsay, Facultés des Sciences, 91400 Orsay}}

\begin{document}

\maketitle

\begin{abstract}
In this text we establish an \textit{a priori} estimate for arbitrary \textit{geometric maximal operator} in the plane. Precisely we associate to any family of rectangles $\mathcal{B}$ a geometric quantity $\lambda_{[\mathcal{B}]}$ called its \textit{analytic split} and satisfying $\log( \lambda_{[\mathcal{B}]}) \lesssim_p \| M_{\mathcal{B}} \|_p^p$ for all $1 < p < \infty$, where $M_\mathcal{B}$ is the Hardy-Littlewood type maximal operator associated to the family $\mathcal{B}$. We give then two applications in order to illustrate it. To begin with, this estimate allows us to classify the $L^p(\mathbb{R}^2)$ behavior of \textit{rarefied directional bases}. As a second application, we prove that the basis $\mathcal{B}$ generated by rectangle whose eccentricity and orientation are of the form $$\left( e_r, \omega_r \right) = \left( \frac{1}{n} , \sin(n) \frac{\pi}{4} \right) $$ for some $n \in \mathbb{N}$, yields a geometric maximal operator $M_\mathcal{B}$ which is unbounded on $L^p(\mathbb{R}^2)$ for any $1 < p < \infty$.
\end{abstract}

\section{Introduction}

In \cite{BATEMANKATZ}, Bateman and Katz developed a powerful method to study the directional maximal operator associated to a Cantor set of directions. In particular they proved that this operator is unbounded on $L^p(\mathbb{R}^2)$ for any $1 \leq p < \infty$. Then in \cite{BATEMAN} - proving the converse of a result due to Alfonseca \cite{ALFONSECA} and developing further the ideas in \cite{BATEMANKATZ} - Bateman classified the $L^p(\mathbb{R}^2)$ behavior of any directional maximal operator in the plane : he proved that a directional maximal operator is either bounded on $L^p(\mathbb{R}^2)$ for any $1 < p < \infty$ or either unbounded for any $1 < p < \infty$. In this text, we pursue the programm initiated in \cite{GAUVAN} which consists in studying \textit{geometric maximal operators} which are not \textit{directional}. It appears than geometric maximal operators are more general than directional maximal operators and their study requires to focus on the interactions between the coupling \textit{eccentricity/orientation} for a family of rectangles. Our main result is the construction of so-called \textit{Kakeya-type sets} for an arbitrary geometric maximal operator which gives an \textit{a priori} bound on their $L^p(\mathbb{R}^2)$-norm in the same spirit than in \cite{BATEMAN} ; we will derive two applications of this estimate to illustrate it.

\subsection*{Definitions}

We work in the euclidean plane $\mathbb{R}^2$ ; if $u$ is a measurable subset we denote by $|u|$ its Lebesgue measure. We denote by $\mathcal{R}$ the collection containing all rectangles of $\mathbb{R}^2$ ; for $r \in \mathcal{R}$ we define its \textit{orientation} as  the angle $\omega_r \in [0,\pi) $ that its longest side makes with the $x$-axis and its \textit{eccentricity} as the ratio $e_r \in (0,1]$ of its shortest side by its longest side.

For an arbitrary non empty family $\mathcal{B}$ contained in $\mathcal{R}$, we define the associated \textit{derivation basis} $\mathcal{B}^*$ by $$\mathcal{B}^* = \left\{ \Vec{t} + h r : \Vec{t} \in \mathbb{R}^2, h>0, r \in \mathcal{B} \right\}.$$ The derivation basis $\mathcal{B}^*$ is simply the smallest collection which is invariant by dilation and translation and that contains $\mathcal{B}$. Without loss of generality, we identify the derivation basis $\mathcal{B}^*$ and any of its generator $\mathcal{B}$.

Our object of interest will be the \textit{geometric maximal operator $M_\mathcal{B}$ generated by $\mathcal{B}$} which is defined as $$M_\mathcal{B}f(x) := \sup_{x \in r \in \mathcal{B}^*}  \frac{1}{|r|} \int_r |f|$$ for any $f \in L_{loc}^{1}(\mathbb{R}^2)$ and $x \in \mathbb{R}^2$. \textbf{Observe that the upper bound is taken on elements of $\mathcal{B}^*$ that contain the point $x$}. The definitions of $\mathcal{B}^*$ and $M_\mathcal{B}$ remain valid when we consider that $\mathcal{B}$ is an arbitrary family composed of open bounded convex sets. For example in this note, for technical reasons and without loss of generality, we will work at some point with parallelograms instead of rectangles.

For $p \in (1, \infty]$ we define as usual the operator norm $\|M_\mathcal{B}\|_p$ of $M_\mathcal{B}$ by $$ \|M_\mathcal{B}\|_p = \sup_{ \|f\|_p = 1} \|M_\mathcal{B}f\|_p.$$ If $\|M_\mathcal{B}\|_p < \infty$ we say that $M_\mathcal{B}$ is \textit{bounded} on $L^p(\mathbb{R}^2)$.  The boundedness of a maximal operator $M_\mathcal{B}$ is related to the geometry that the family $\mathcal{B}$ exhibits. 

\begin{defn}
We will say that the operator $M_\mathcal{B}$ is a \textbf{good operator} when it is bounded on $L^p(\mathbb{R}^2)$ for any $ p > 1$. On the other hand, we say that the operator $M_\mathcal{B}$ is a \textbf{bad operator} when it is unbounded on $L^p(\mathbb{R}^2)$ for any $1 < p < \infty$.
\end{defn}

On the $L^p(\mathbb{R}^2)$ range, to be able to say that a operator $M_\mathcal{B}$ is good or bad is an optimal result. We are going to see that a certain type of geometric maximal operators, namely directional maximal operators, are known to be either good or bad.

\subsection*{Directional maximal operators}

A lot of researches have been done in the case where $\mathcal{B}$ is equal to $\mathcal{R}_\Omega :=  \left\{ r \in \mathcal{R} : \omega_r \in \Omega \right\}$ where $\Omega$ is an arbitrary set of directions in $[0,\pi)$. In other words, $\mathcal{R}_\Omega$ is the set of \textit{all} rectangles whose orientation belongs to $\Omega$. We say that $\mathcal{R}_\Omega$ is a \textit{directional basis} and to alleviate the notation we denote $$M_{\mathcal{R}_\Omega} := M_\Omega.$$ In the literature, the operator $M_\Omega$ is said to be a \textit{directional maximal operator}. The study of those operators goes back at least to Cordoba and Fefferman's article \cite{CORDOBAFEFFERMAN II} in which they use geometric techniques to show that if $\Omega = \left\{ \frac{\pi}{2^k} \right\}_{ k \geq 1}$ then $M_\Omega$ has weak-type $(2,2)$. A year later, using Fourier analysis techniques, Nagel, Stein and Wainger proved in \cite{NSW} that $M_\Omega$ is actually bounded on $L^p(\mathbb{R}^2)$ for any $p > 1$. In \cite{ALFONSECA}, Alfonseca has proved that if the set of direction $\Omega$ is a \textit{lacunary set of finite order} then the operator $M_\Omega$ is bounded on $L^p(\mathbb{R}^2)$ for any $p > 1$. Finally in \cite{BATEMAN}, Bateman proved the converse and so characterized the $L^p(\mathbb{R}^2)$-boundedness of directional operators. Precisely he proved the following Theorem.

\begin{thm}[Bateman]\label{ T : bateman }
Fix an arbitrary set of directions $\Omega \subset [0,\pi)$. The directional maximal operator $M_\Omega$ is either good or bad.
\end{thm}

We invite the reader to look at \cite{BATEMAN} for more details and also \cite{BATEMANKATZ} where Bateman and Katz introduced their method. Hence we know that a set of directions $\Omega$ always yields a directional operator $M_\Omega$ that is either good or bad. Merging the vocabulary, we use the following definition.

\begin{defn}
We say that a set of directions $\Omega$ is a good set of directions when $M_\Omega$ is good and that it is a  bad set of directions when $M_\Omega$ is bad.
\end{defn}

The notion of good/bad is perfectly understood for a set of directions $\Omega$ and the associated directional operator $M_\Omega$. To say it bluntly, $\Omega$ is a good set of directions if and only if it can be included in a finite union of lacunary sets of finite order. If this is not possible, then $\Omega$ is a bad set of directions ; see \cite{BATEMAN}. We now turn attention to maximal operator which are not directional.

\subsection*{Geometric maximal operators}

\textbf{In this text, we will focus on geometric maximal operator which are not directional}. In \cite{GAUVAN}, we have considered the following type of basis : for $a,b > 0$ arbitrary, denote by $\mathcal{B}_{a,b}$ the basis generated by rectangles $r$ whose eccentricity and orientation are of the form $$\left(e_r , \omega_r \right) = \left( \frac{1}{n^a}, \frac{\pi}{4n^b} \right)$$ for some $n \in \mathbb{N}^*$. Obviously the basis $\mathcal{B}_{a,b}$ is not a directional basis ; denoting by $M_{a,b}$ the geometric maximal operator associated we proved the following Theorem.

\begin{thm}[Gauvan]
If $a \leq b$ then $M_{a,b}$ is a good operator. If not then $M_{a,b}$ is a bad operator.
\end{thm}

To prove this Theorem, we developed geometric estimates in order to fully exploit \textit{generalized Perron trees} as constructed in \cite{KATHRYN JAN} by Hare and Rönning. However, it appears that generalized Perron trees are \textit{ad hoc} constructions that can only made in specific situations.

\subsection*{Results}

Our main result is an \textit{a priori} estimate in the same spirit than one of the main result of \cite{BATEMAN}. Precisely, to any family $\mathcal{B}$ contained in $\mathcal{R}$ we associate a geometric quantity $\lambda_{[\mathcal{B}]} \in \mathbb{N} \cup \left\{\infty \right\}$ that we call \textit{analytic split} of $\mathcal{B}$. Loosely speaking, the analytic split $ \lambda_{[\mathcal{B}]}$ indicates if $\mathcal{B}$ contains a lot of rectangles in terms of \textit{orientation} and \textit{eccentricity}. We prove then the following Theorem.

\begin{thm}\label{ T : main theorem}
For any family $\mathcal{B}$ and any $1 < p < \infty$ we have $$A_p \times \log(\lambda_{[\mathcal{B}]}) \leq \| M_\mathcal{B} \|_p^p$$ where $A_p$ is a constant only depending on $p$.
\end{thm}

An important feature of this inequality is that we do not make any assumption on the family $\mathcal{B}$. Observe that the analytic split of a family $\mathcal{B}$ indicates if the family $\mathcal{B}$ is large \textit{i.e.} if $M_\mathcal{B}$ is an operator with large $L^p(\mathbb{R}^2)$-norms. In regards of the study of geometric maximal operators, Theorem \ref{ T : main theorem} gives a \textit{concrete} and \textit{a priori} lower bound on the $L^p(\mathbb{R}^2)$ norm of $M_\mathcal{B}$. \textbf{We insist on the fact that this estimate is concrete since the analytic split is not an abstract quantity associated to $\mathcal{B}$ but has strong a \textit{geometric interpretation}}. No such results was previously known for geometric maximal operators and we give two applications in order to illustrate it. The following Theorem allows us to classify the $L^p(\mathbb{R}^2)$ behavior of \textit{rarefied directional bases}.

\begin{thm}\label{T:app1}
Fix any bad set of directions $\Omega \subset [0,\frac{\pi}{4})$ and let $\mathcal{B} \subset \mathcal{R}_\Omega$ be a family satisfying for any $\omega \in \Omega$ $$\inf_{ r \in \mathcal{B}, \omega_r = \omega} e_r = 0.$$ In this case the operator $M_\mathcal{B}$ is also a bad operator.
\end{thm}

A basis $\mathcal{B}$ satisfying the condition of Theorem \ref{T:app1} is said to be a \textit{rarefaction of the directional basis $\mathcal{R}_\Omega$}. Observe that since we have $\mathcal{B} \subset \mathcal{R}_\Omega$ we have the trivial pointwise estimate $$M_\mathcal{B} \leq M_\Omega.$$ Hence - trivially - we have $\|M_\mathcal{B}\|_p < \infty$ if $\|M_\Omega\|_p < \infty$. Surprisingly, Theorem \ref{T:app1} states that the conserve is also true \textit{i.e.} we have $\|M_\mathcal{B}\|_p = \infty$ if $\|M_\Omega\|_p = \infty$. This discussion gives a classification of rarefied directional maximal operator.

We give another application of Theorem \ref{ T : main theorem} : for $n \in \mathbb{N}^*$ let $r_n \in \mathcal{R}$ be a rectangle whose eccentricity and orientation is of the form $$\left( e_{r_n}, \omega_{r_n} \right) = \left( \frac{1}{n}, \sin(n)\frac{\pi}{4} \right).$$ Consider then the basis $\mathcal{B}_e$ generated by the rectangles $\{r_n\}_{n \geq 1}$ ; we have the following Theorem.

\begin{thm}\label{T :app2}
The operator $M_{\mathcal{B}_e}$ is a bad operator.
\end{thm}

It seems that one needs to obtain a Theorem as least as general as Theorem \ref{ T : main theorem} in order to tackle easily a basis such that $\mathcal{B}_e$. Hopefully Theorems \ref{T:app1} and \ref{T :app2} illustrate the implications of Theorem \ref{ T : main theorem}.

\subsection*{Plan}

Most of this text is dedicated to the proof of Theorem \ref{ T : main theorem} ; it is organized as follow. To begin with, we will explain how we can \textit{discretize} the collection $\mathcal{R}$ which will allow us to precisely define the analytic split of a family $\mathcal{B}$, see sections \ref{S:Def of}, \ref{ S : structure } and \ref{ S : ana split }. Then in sections \ref{S : tools} and \ref{ S : An Estimate of Bateman }, we introduce the notion of \textit{Kakeya-type sets} and recall how Bateman constructed them in \cite{BATEMAN}. Finally we develop important geometric estimates in section \ref{S:geometric estimate} and we prove Theorem \ref{ T : main theorem} in section \ref{ S : proof thm main }. The last two sections are devoted to the applications of Theorem \ref{ T : main theorem}.

\section*{Acknowledgments}

I warmly thank Laurent Moonens and Emmanuel Russ for their kind advices.

\section{Definition of $\mathcal{T}$}\label{S:Def of}

%%%%%%%%%%%%%%%%%%%%%
%%%%%%%%%%%%%%%%%%%%%
%%%%%%%%%%%%%%%%%%%%%
%%%%%%%%%%%%%%%%%%%%%
%%%%%%%%%%%%%%%%%%%%%
%%%%%%%%%%%%%%%%%%%%%
%%%%%%%%%%%%%%%%%%%%%
%%%%%%%%%%%%%%%%%%%%%
%%%%%%%%%%%%%%%%%%%%%
\begin{figure}[h!]
\centering
\includegraphics[scale=1]{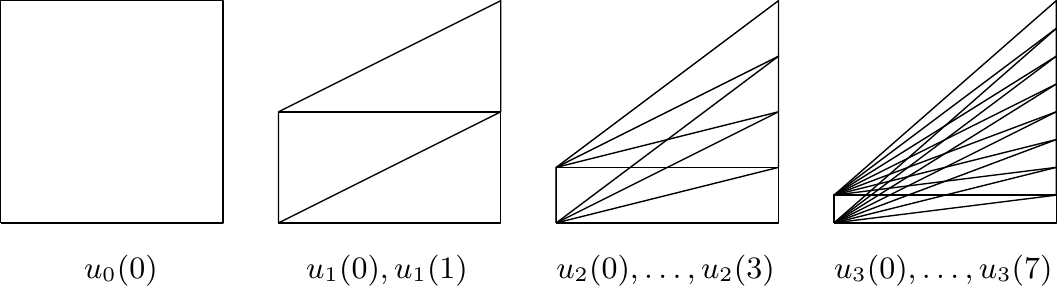}
\caption{A representation of the first element of $\mathcal{T}$.}  
\end{figure}
%%%%%%%%%%%%%%%%%%%%%
%%%%%%%%%%%%%%%%%%%%%
%%%%%%%%%%%%%%%%%%%%%
%%%%%%%%%%%%%%%%%%%%%
%%%%%%%%%%%%%%%%%%%%%
%%%%%%%%%%%%%%%%%%%%%
%%%%%%%%%%%%%%%%%%%%%
%%%%%%%%%%%%%%%%%%%%%
%%%%%%%%%%%%%%%%%%%%%

Instead of working with rectangles we will consider that our family $\mathcal{B}$ is included in the collection $\mathcal{T}$ composed of \textit{pulled-out parallelograms} which is defined as follow. For $n \geq 0$ and $0 \leq k \leq 2^n-1$ consider the parallelogram $u_n(k)$  whose vertices are the points $(0,0),(0, \frac{1}{2^n}),(1,\frac{k-1}{2^n})$ and $(1,\frac{k}{2^n})$. We say that $u_n(k)$ is a \textit{pulled-out parallelogram of scale $n$} and we define the collection $\mathcal{T}$ as $$\mathcal{T} = \left\{ u_n(k) : n \geq 0, 0 \leq k \leq 2^n-1 \right\}.$$ \textbf{Morally, the parallelogram $u_n(k)$ should be thought as a rectangle whose eccentricity and orientation are $$\left( e_{u_n(k)}, \omega_{u_n(k)} \right) = \left(  \frac{1}{2^n},  \frac{k}{2^n}\frac{\pi}{4} \right).$$ }The following proposition precises that we do not lose information if we consider that our family are contained in $\mathcal{T}$ and not in $\mathcal{R}$. We won't prove it since this kind of reduction is well known in the literature, see Bateman \cite{BATEMAN} or Alfonseca \cite{ALFONSECA} for examples.

\begin{prp}\label{ P : approx }
Fix an arbitrary family $\mathcal{B}$ in $\mathcal{R}$. Without loss of generality, we can suppose that we have $\{ \omega_r : r \in \mathcal{B} \} \subset [0,\frac{\pi}{4}).$ There exists a family $\mathcal{B}_a$ contained in $\mathcal{T}$ satisfying the following inequality $$\frac{1}{C_d} \times M_{\mathcal{B}_a}  \leq M_\mathcal{B} \leq {C_d} \times M_{\mathcal{B}_a}$$ where $C_d = C_2$ is a constant only depending on the dimension $d = 2$.
\end{prp}

\textbf{In regards of the $L^p(\mathbb{R}^2)$-norm, the maximal operator $M_\mathcal{B}$ and $M_{\mathcal{B}_a}$ have the same behavior and so we will identify $\mathcal{B}$ and $\mathcal{B}_a$. Hence, unless stated otherwise, we will always supposed that our family $\mathcal{B}$ is now contained in $\mathcal{T}$.} We give an example : consider the family $\mathcal{B} = \mathcal{R}_{\{0\}}$. In this case, we denote the operator $M_{\{0\}}$ by $M_S$ : in the the literature, $M_S$ is called the \textit{strong} maximal operator. We would like an explicit pointwise approximation of $M_S$ by an operator $M_{\mathcal{B}_0}$ where $\mathcal{B}_0$ is a family in $\mathcal{T}$, as announced in Proposition \ref{ P : approx }. Observe that the family $\mathcal{B}_0$ defined as $$\mathcal{B}_0 := \{ u_n(0) \in \mathcal{T} : n \geq 0 \}$$ satisfies Proposition \ref{ P : approx } in this case ; precisely one has for any $f$ locally integrable and $x \in \mathbb{R}^2$ $$M_{\mathcal{B}_0 } f(x) \leq M_S f(x) \leq 2 M_{\mathcal{B}_0 } f(x).$$

\section{Structure of $\mathcal{T}$}\label{ S : structure }

The collection of $\mathcal{T}$ has a natural structure of binary tree and we develop a vocabulary adapted to this structure.

%%%%%%%%%%%%%%%%%%%%%
%%%%%%%%%%%%%%%%%%%%%
%%%%%%%%%%%%%%%%%%%%%
%%%%%%%%%%%%%%%%%%%%%
%%%%%%%%%%%%%%%%%%%%%
%%%%%%%%%%%%%%%%%%%%%
%%%%%%%%%%%%%%%%%%%%%
%%%%%%%%%%%%%%%%%%%%%
%%%%%%%%%%%%%%%%%%%%%
\begin{figure}[h!]
\centering
\includegraphics[scale=1]{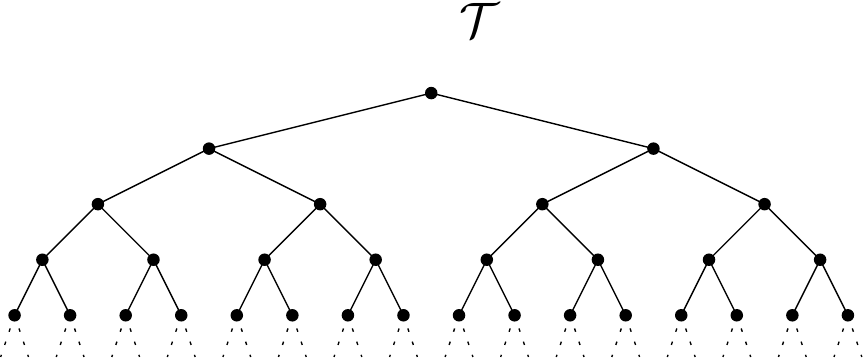}
\caption{A representation of the first element of $\mathcal{T}$.}  
\end{figure}
%%%%%%%%%%%%%%%%%%%%%
%%%%%%%%%%%%%%%%%%%%%
%%%%%%%%%%%%%%%%%%%%%
%%%%%%%%%%%%%%%%%%%%%
%%%%%%%%%%%%%%%%%%%%%
%%%%%%%%%%%%%%%%%%%%%
%%%%%%%%%%%%%%%%%%%%%
%%%%%%%%%%%%%%%%%%%%%
%%%%%%%%%%%%%%%%%%%%%

%%%%%%%%%%%%%%%%%%%%%
%%%%%%%%%%%%%%%%%%%%%
%%%%%%%%%%%%%%%%%%%%%
%%%%%%%%%%%%%%%%%%%%%
%%%%%%%%%%%%%%%%%%%%%
%%%%%%%%%%%%%%%%%%%%%
%%%%%%%%%%%%%%%%%%%%%
%%%%%%%%%%%%%%%%%%%%%
%%%%%%%%%%%%%%%%%%%%%
\begin{figure}[h!]
\centering
\includegraphics[scale=1]{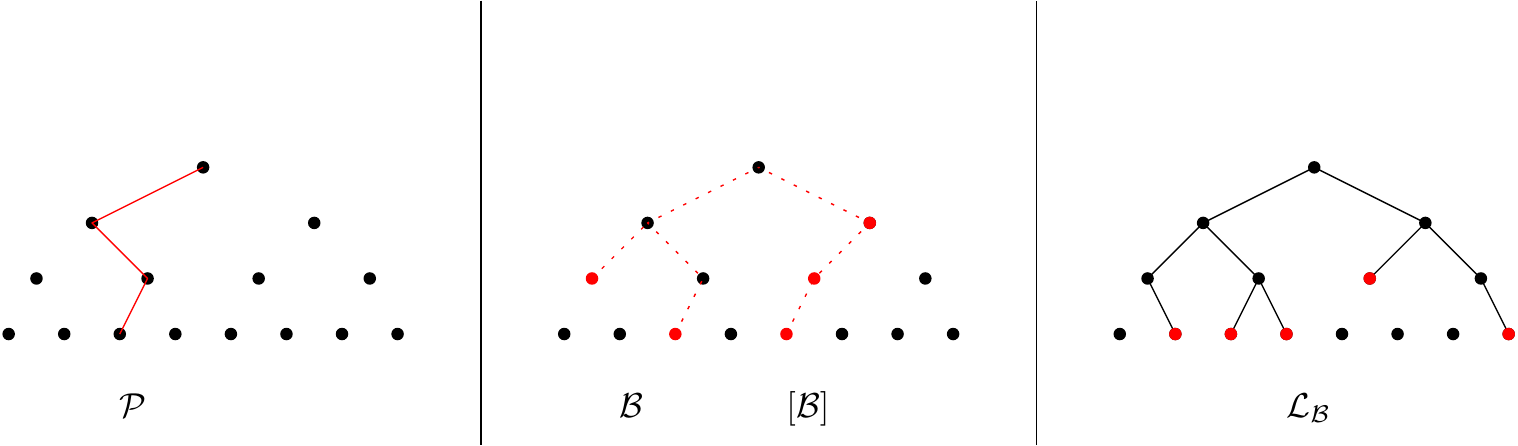}
\caption{From the left to the right : a path $\mathcal{P}$, a family $\mathcal{B}$ and the tree it generates $[\mathcal{B}]$ and the leaves of tree.}  
\end{figure}
%%%%%%%%%%%%%%%%%%%%%
%%%%%%%%%%%%%%%%%%%%%
%%%%%%%%%%%%%%%%%%%%%
%%%%%%%%%%%%%%%%%%%%%
%%%%%%%%%%%%%%%%%%%%%
%%%%%%%%%%%%%%%%%%%%%
%%%%%%%%%%%%%%%%%%%%%
%%%%%%%%%%%%%%%%%%%%%
%%%%%%%%%%%%%%%%%%%%%

\subsection*{Parent and children}

For any $u \in \mathcal{T}$ of scale $n \geq 1$, there exist a unique $u_f \in \mathcal{T}$ of scale $n-1$ such that  $u \subset u_f$. We say that $u_f$ is the \textit{parent} of $u$. In the same fashion, observe that there are only two elements $u_h,u_l \in \mathcal{T}$ of scale $n+1$ such that $u_h,u_l \subset u$. We say that $u_h$ and $u_l$ are the \textit{children} of $u$. Observe that $u \in \mathcal{T}$ is the child of $v \in \mathcal{T}$ if and only if $u \subset v$ and $2|u| = |v|$ : we will often use those two conditions.

\subsection*{Path}

We say that a sequence (finite or infinite) $\left\{ u_i \right\}_{i \in \mathbb{N}}  \subset \mathcal{T}$ is a \textit{path} if it satisfies $u_{i+1} \subset u_i$ and $2|u_{i+1}| = |u_i| $ for any $i$ \textit{i.e.} if $u_i$ is the parent of $u_{i+1}$ for any $i$. Different situations can occur. A finite path $\mathcal{P}$ has a first element $u$ and a last element $v$ (defined in a obvious fashion) and we will write $\mathcal{P}_{u,v} := \mathcal{P}$. On the other hand, an infinite path $\mathcal{P}$ has no endpoint.

\subsection*{Tree}

For any family $\mathcal{B}$ contained in $\mathcal{T}$, there is a unique parallelogram $r \in \mathcal{T}$ such that any $u \in \mathcal{B}$ is included in $r$ and $|r|$ is minimal. We say that this element $r_\mathcal{B} := r$ is the \textit{root} of $\mathcal{B}$ and we define the set $[\mathcal{B}]$ as $$[\mathcal{B}] := \left\{ u \in \mathcal{T} : \exists v \in \mathcal{B}, v \subset u \subset r_\mathcal{B} \right\}.$$ A subset of  $\mathcal{T}$ of the form $[\mathcal{B}]$ is called a \textit{tree} generated by $\mathcal{B}$.

\subsection*{Leaf}

We define the set $L_\mathcal{B}$ as $$\mathcal{L}_\mathcal{B} = \left\{ u \in \mathcal{B} : \forall v \in \mathcal{B}, v \subset u \Rightarrow v = u \right\}.$$ An element of $\mathcal{L}_\mathcal{B}$ is called a \textit{leaf} of $\mathcal{B}$. Observe that for any $\mathcal{B}$ in $\mathcal{T}$ we have $[\mathcal{B}] = [\mathcal{L}_\mathcal{B}]$ and also $ \mathcal{L}_\mathcal{B} = \mathcal{L}_{[\mathcal{B}]}$. The first identity says that the leaves of a tree $[\mathcal{B}]$ can be seen as the minimal set that generates $[\mathcal{B}]$. The second identity states that $[\mathcal{B}]$ is not bigger than $\mathcal{B}$ in the sense that it does not have more leaves. If $\mathcal{P}$ is an infinite path, we have by definition $\mathcal{L}_\mathcal{P} = \emptyset$.

\subsection*{Structural disposition}

Let $\mathcal{B}$ be an arbitrary family in $\mathcal{T}$ and let $r$ be the root of $[\mathcal{B}]$. We fix an arbitrary element $\Tilde{r}$ in $\mathcal{T}$ and we consider the family $\Tilde{\mathcal{B}}$ defined as follow : the family $\Tilde{\mathcal{B}}$ has the same disposition than $\mathcal{B}$ in $\mathcal{T}$ but $[\Tilde{\mathcal{B}}]$ is rooted at $\Tilde{r}$. In order to formulate it precisely consider the unique bijective linear map with positive determinant $L : \mathbb{R}^2 \rightarrow \mathbb{R}^2 $ such that $L(r) = \Tilde{r}$ and define the family $\Tilde{\mathcal{B}}$ as $$\Tilde{\mathcal{B}} := \left\{ L(u) : u \in \mathcal{B} \right\} \subset \mathcal{T}.$$ Now, it is routine to show that we have for any $f \in L^1_{\mathrm{loc}}(\mathbb{R}^2)$ $$M_{\Tilde{\mathcal{B}}} f = \frac{1}{|\det(L)|} \times M_\mathcal{B} (f \circ L)$$ and so we have $\|M_{\Tilde{\mathcal{B}}}\|_p = \|M_{{\mathcal{B}}}\|_p$ for any $1 < p < \infty$. Hence, what truly matters when considering a family $\mathcal{B}$ contained in $\mathcal{T}$ is not its \textit{absolute position} in the tree $\mathcal{T}$ but its \textit{structural disposition} in the binary tree.

\section{Analytic split}\label{ S : ana split }

We associate to any family $\mathcal{B}$ included in $\mathcal{T}$ a natural number $\lambda_{[\mathcal{B}]} \in \mathbb{N} \cup \{ \infty \}$ that we call \textit{analytic split} ; its definition relies on specific trees in $\mathcal{T}$, namely \textit{fig trees}.

%%%%%%%%%%%%%%%%%%%%%
%%%%%%%%%%%%%%%%%%%%%
%%%%%%%%%%%%%%%%%%%%%
%%%%%%%%%%%%%%%%%%%%%
%%%%%%%%%%%%%%%%%%%%%
%%%%%%%%%%%%%%%%%%%%%
%%%%%%%%%%%%%%%%%%%%%
%%%%%%%%%%%%%%%%%%%%%
%%%%%%%%%%%%%%%%%%%%%
\begin{figure}[h!]
\centering
\includegraphics[scale=0.9]{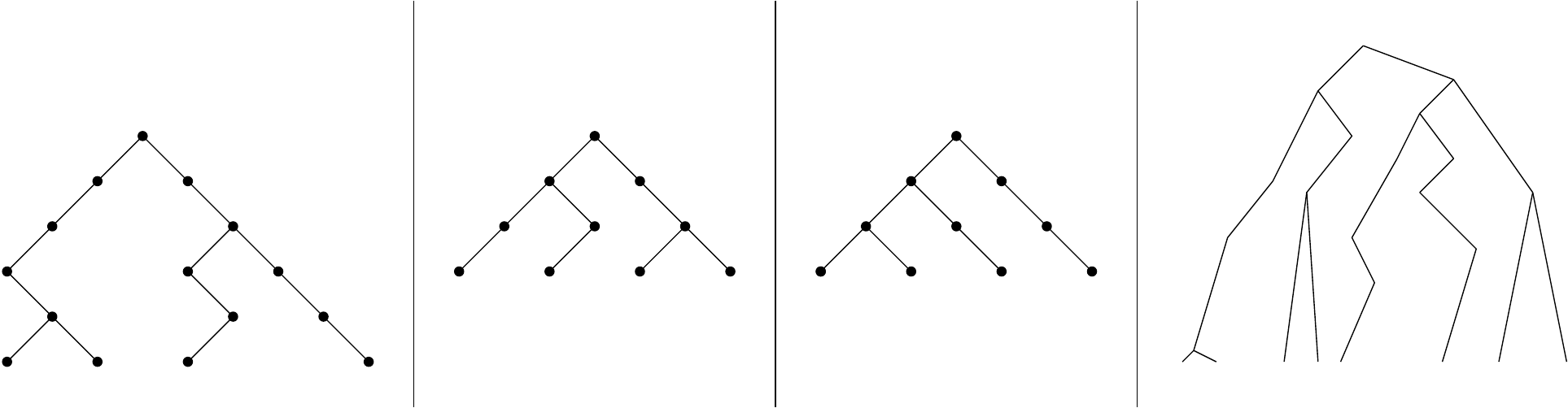}
\caption{The first two tree are fig trees of scale $2$, the third tree is not a fig tree and the last tree is a fig tree of scale $3$.}  
\end{figure}
%%%%%%%%%%%%%%%%%%%%%
%%%%%%%%%%%%%%%%%%%%%
%%%%%%%%%%%%%%%%%%%%%
%%%%%%%%%%%%%%%%%%%%%
%%%%%%%%%%%%%%%%%%%%%
%%%%%%%%%%%%%%%%%%%%%
%%%%%%%%%%%%%%%%%%%%%
%%%%%%%%%%%%%%%%%%%%%
%%%%%%%%%%%%%%%%%%%%%

\subsection*{Boundary of $\left[\mathcal{B}\right]$ and splitting number}

For any tree $[\mathcal{B}]$, we define its \textit{boundary} $\partial [\mathcal{B}]$ as the set of path in $[\mathcal{B}]$ that are maximal for the inclusion \textit{i.e.} $\mathcal{P} \in \partial [\mathcal{B}]$ if and only if $\mathcal{P}$ is a path included in $[\mathcal{B}]$ such that if $\mathcal{P}' \subset [\mathcal{B}]$ is a path that contains $\mathcal{P}$ then $\mathcal{P} = \mathcal{P}'$. For any tree $[\mathcal{B}]$ and path $\mathcal{P} \in \partial[\mathcal{B}]$ we define the \textit{splitting number of $\mathcal{P}$ relatively to $[\mathcal{B}]$} as $$s_{\mathcal{P},[\mathcal{B}]} := \# \left\{ u \in [\mathcal{B}] \setminus \mathcal{P} : \exists v \in \mathcal{P}, u \subset v, 2|u| = |v| \right\}.$$ Observe that the splitting number of a path $\mathcal{P}$ is defined relatively to a tree $[\mathcal{B}]$ \textit{i.e.} we might have $s_{\mathcal{P},[\mathcal{B}]} \neq s_{\mathcal{P},[\mathcal{C}]}$ for different trees $\mathcal{B}$ and $\mathcal{C}$.

\subsection*{Fig trees $\left[\mathcal{F} \right]$}

We say that a tree $[\mathcal{F}]$ is a \textit{fig tree of scale $n$ and height $h$} when  \begin{itemize}
    \item $[\mathcal{F}]$ is finite and $ \#\partial[\mathcal{F}] = 2^n $
    \item for any $\mathcal{P} \in \partial[\mathcal{F}]$ we have $ s_{\mathcal{P},[\mathcal{F}]} = n$ and $\# \mathcal{P} = h$.
\end{itemize} Observe that by construction we always have $h \geq n$. A basic example of fig tree of scale $n$ is the tree $[\mathcal{T}_n]$ defined as $[\mathcal{T}_n] = \left\{ u \in \mathcal{T} : |u| \geq \frac{1}{2^n} \right\}$. In this case, the height of $[\mathcal{T}_n]$ is $n$ ; however this is the only fig tree satisfying this. One may see a fig tree $[\mathcal{F}]$ of scale $n$ as a uniformly stretched version of $[\mathcal{T}_n]$.

\subsection*{Analytic split of $\mathcal{B}$}

We define the \textit{analytic split $\lambda_{[\mathcal{B}]}$ of a tree $[\mathcal{B}]$} as the integer $n$ such that $[\mathcal{B}]$ contains a fig tree $[\mathcal{F}]$ of scale $n$ and do not contains any fig tree of scale $n+1$. In the case where $[\mathcal{B}]$ contains fig trees of arbitrary high scale, we set $\lambda_{[\mathcal{B}]} = \infty$. More generally for any family $\mathcal{B}$ contained in $\mathcal{T}$ (\textit{i.e.} when $\mathcal{B}$ is not necessarily a tree), we define its analytic split as $$\lambda_\mathcal{B} := \lambda_{[\mathcal{B}]}.$$ \textbf{Hence by definition, the analytic split of a family $\mathcal{B}$ is the same as the analytic split of the tree $[\mathcal{B}]$. Observe that thanks to Theorem \ref{ T : main theorem} this definition is pertinent.}

\section{Kakeya-type sets}\label{S : tools}

We detail how we can construct a set $A$ with elements of $\mathcal{B}^*$ that gives non trivial lower bound on $\|M_\mathcal{B} \|_p$ for any $1 < p < \infty$. We say that a maximal operator $M_\mathcal{B}$ \textit{admits a Kakeya-type set $A \subset \mathbb{R}^2$ of level $(\eta,\epsilon)$} with $\epsilon,\eta > 0$ when we have $$|A| \leq \epsilon \times \left|\left\{ M_\mathcal{B}\mathbb{1}_A > \eta \right\}\right|.$$ In this case, for any $p > 1$ we have $$\|M_\mathcal{B}\|_p  \geq  \eta{\epsilon^{-\frac{1}{p}} }.$$ Indeed, we have $\int (M_\mathcal{B}\mathbb{1}_{A})^p \geq \eta^p \epsilon^{-1}|A|$ ; since $|A| = \|\mathbb{1}_A\|_p^p$.

\begin{prp}
If $M_\mathcal{B}$ admits a Kakeya-type set of level $\left(\eta,\epsilon\right)$ then for any $1 < p < \infty$ we have $$\|M_\mathcal{B}\|_p  \geq  \eta{\epsilon^{-\frac{1}{p}} }.$$
\end{prp}

Formally one can construct interesting Kakeya-type sets for $M_\mathcal{B}$ with elements of $\mathcal{B}^*$ as follow. Suppose there is a collection $ \left\{ p_i \right\}_{i \in I} \subset \mathcal{B}^*$  such that for each $i \in I$ there is a subset $s_i \subset p_i$ satisfying $|s_i| \geq \eta|p_i|$ and $$\left|\bigcup_{i \in I} s_i\right| < \epsilon \left|\bigcup_{i \in I } p_i\right|.$$ In this case, the set $A := \bigcup_{i \in I} s_i$ is a Kakeya-type set of level $(\eta,\epsilon)$. Indeed, we have the following inclusion $$ \bigcup_{i \in I} p_i \subset \left\{ M_\mathcal{B}\mathbb{1}_{ A} > \eta \right\}$$ because $p_i \in \mathcal{B}^*$ for any $i \in I$ and so $|A| \leq \epsilon \left|\left\{ M_\mathcal{B}\mathbb{1}_{ A } > \eta \right\}\right|$.

\section{Bateman's construction}\label{ S : An Estimate of Bateman }

In \cite{BATEMAN}, Bateman proves the following Theorem \ref{ T : main thm bateman } by making an explicit construction of a Kakeya-type set of the desired level. We will recall how he achieves the construction of this set since we will use it in order to prove Theorem \ref{ T : main theorem}.

\begin{thm}[Bateman's construction \cite{BATEMAN}]\label{ T : main thm bateman }
Suppose that $[\mathcal{F}]$ is a fig tree of scale $n$ and height $h$. In this case the maximal operator $M_{[\mathcal{F}]}$ admits a Kakeya-type set of level $$\left(\frac{1}{4}, C \log(n)^{-1}\right) \simeq \left(\frac{1}{4},\log(n)^{-1}\right).$$
\end{thm}

We fix an arbitrary fig tree $[\mathcal{F}]$ of scale $n$ and height $h$ rooted at $u_0(0)$ ; we are looking for a Kakeya-type set - that we will denote $A_1$ - of level $$\left(\frac{1}{4}, C \log(n)^{-1}\right).$$ Bateman constructs this Kakeya-type set $A_1$ as a realisation of a random set that we denote - in the same fashion, $A_1(\omega)$ - this is done in three steps.

%%%%%%%%%%%%%%%%%%%%%
%%%%%%%%%%%%%%%%%%%%%
%%%%%%%%%%%%%%%%%%%%%
%%%%%%%%%%%%%%%%%%%%%
%%%%%%%%%%%%%%%%%%%%%
%%%%%%%%%%%%%%%%%%%%%
%%%%%%%%%%%%%%%%%%%%%
%%%%%%%%%%%%%%%%%%%%%
%%%%%%%%%%%%%%%%%%%%%
\begin{figure}[h!]
\centering
\includegraphics[scale=0.8]{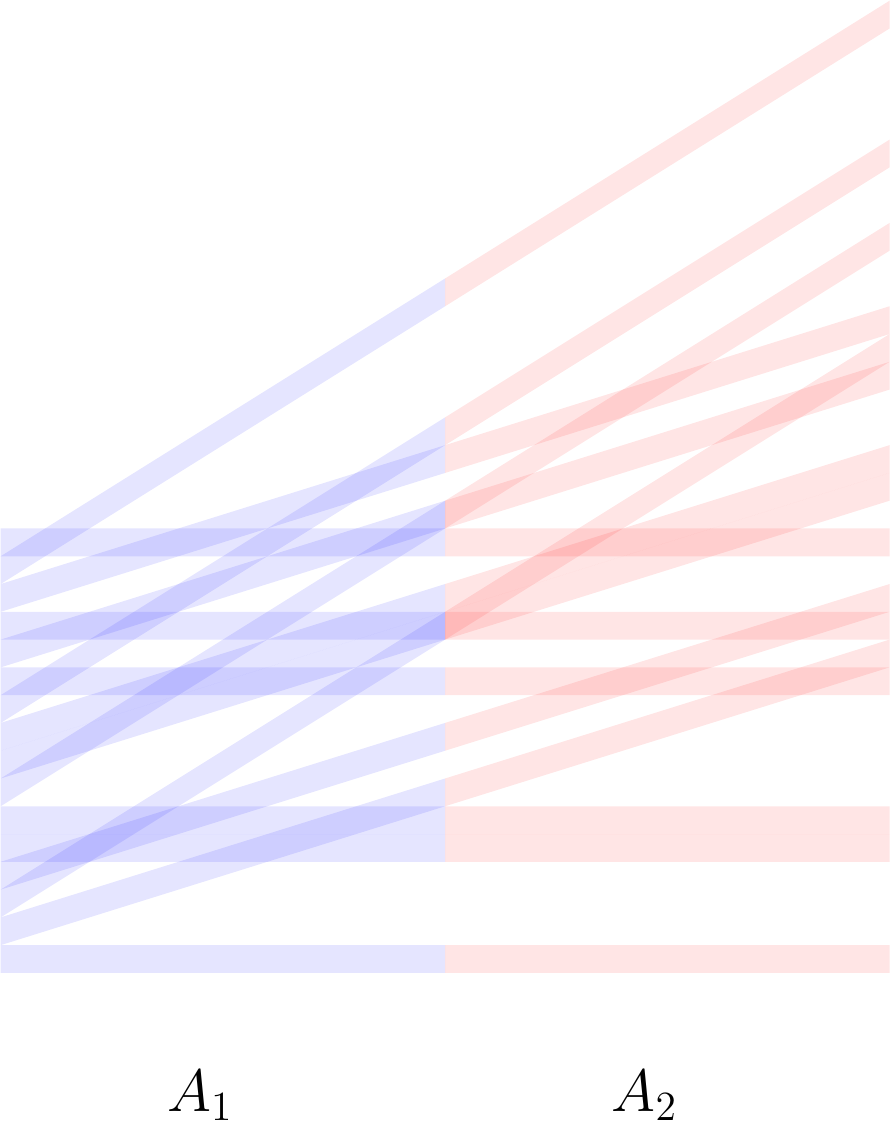}
\caption{With positive probability, the random sets $A_1$ and $A_2$ satisfies $|A_2| \gtrsim \log(n) |A_1|$. }  
\end{figure}
%%%%%%%%%%%%%%%%%%%%%
%%%%%%%%%%%%%%%%%%%%%
%%%%%%%%%%%%%%%%%%%%%
%%%%%%%%%%%%%%%%%%%%%
%%%%%%%%%%%%%%%%%%%%%
%%%%%%%%%%%%%%%%%%%%%
%%%%%%%%%%%%%%%%%%%%%
%%%%%%%%%%%%%%%%%%%%%
%%%%%%%%%%%%%%%%%%%%%

\subsection*{Step 1 : construction of $A_2(\omega)$}

For $u \in \mathcal{T}^*$, we will denote by $u'$ the parallelogram $u$ but shifted of one unit length on the right along its orientation. We fix a $2^h$ mutually independent random variables $$ r_k : (\Omega,\mathbb{P}) \rightarrow \mathcal{L}_{[\mathcal{F}]}$$ who are uniformly distributed in the set $\mathcal{L}_{[\mathcal{F}]}$ \textit{i.e.} for any $k \leq 2^h$ and any $u \in \mathcal{L}_{[\mathcal{F}]}$ we have $$\mathbb{P}(r_k = u) = 2^{-n}.$$ We define then the random set $A$ as $$A = \bigcup_{ k \leq 2^h} \Vec{t}_k + \left( r_k \cup r_k' \right) $$ where $\Vec{t}_k = (0,\frac{k-1}{2^h})$ is a deterministic vector. Define also the first and second halves of $A$ as $$A_1 = \bigcup_{ k \leq 2^h}  \left( \Vec{t}_k + r_k \right) $$ and $$ A_2 = \bigcup_{ k \leq 2^h} \left( \Vec{t}_k + r_k' \right) .$$

\subsection*{Step 2 : Bateman's estimate}

We state Bateman's main result in \cite{BATEMAN} which quantify to which point $|A_2|$ is bigger than $|A_1|$.

\begin{thm}\label{T : Main theorem Bateman}
We have $\mathbb{P}\left(\left|A_2\right| \geq \frac{\log(n)}{ C}|A_1|\right) > 0$. Here $C$ is an absolute constant.
\end{thm}

The proof of this Theorem is difficult. It involves fine geometric estimates, percolation theory and the use of the so-called notion of \textit{stickiness} of thin tubes of the euclidean plane. We refer to \cite{BATEMAN} for its proof and for more information but we would suggest to take a look at \cite{BATEMANKATZ} first. Indeed, in \cite{BATEMANKATZ}, Bateman and Katz built a scheme of proof that is similar to the one in \cite{BATEMAN} but in a simpler setting.

\subsection*{Step 3 : the set $A_1$ is a Kakeya-type set of level $\simeq (\frac{1}{4}, \log(n)^{-1})$}

With positive probability the set $A_1$ is a Kakeya-type set of level $(\frac{1}{4}, C^2\log(n)^{-1})$ for $M_{[\mathcal{F}]}$. Indeed, pick any realisation $\omega \in  \left\{ \left|A_2\right| \geq \frac{ \log(n)}{ C} |A_1| \right\}$ and we show that $A_1 := A_1(\omega)$ is a Kakeya-type set of the desired level. Observe that by construction, for any $x \in \Vec{t}_k + r_k(\omega)' := \Vec{t}_k + r_k'$, we have $$\frac{1}{|\Vec{t}_k + 2r_k'|} \int_{\Vec{t}_k + 2r_k'} \mathbb{1}_{A_1}(y)dy > \frac{| \left\{\Vec{t}_k + 2r_k' \right\} \cap \left\{\Vec{t}_k + r_k \right\}|}{4|r_k|} = \frac{1}{4} $$ and so $$A_2 \subset \left\{ M_{[\mathcal{F}]}\mathbb{1}_{A_1} > \frac{1}{4} \right\}.$$ Since we also have $|A_2| \geq \frac{\log(n)}{C} |A_1|$ this shows that $A_1$ is a Kakeya-type set of level $(\frac{1}{4}, C\log(n)^{-1})$.

\section{Geometric estimates}\label{S:geometric estimate}

We need different geometric estimates in order to prove Theorem \ref{ T : main theorem}. We start with geometric estimates on $\mathbb{R}$ which will help us to prove geometric estimates on $\mathbb{R}^2$. Finally we prove a geometric estimate on $\mathbb{R}^2$ involving geometric maximal operators that is crucial.

\subsection*{Geometric estimates on $\mathbb{R}$}

If $I$ is a bounded interval on $\mathbb{R}$ and $\tau >0$ we denote by $\tau I$ the interval that has the same center as $I$ and $\tau$ times its length \textit{i.e.} $ \left|\tau I \right| = \tau \left| I \right|$. The following lemma can be found in \cite{AUSTIN}.

\begin{lemma}[Austin's covering lemma]
Let $\{ I_\alpha \}_{\alpha \in A}$ a finite family of bounded intervals on $\mathbb{R}$. There is a disjoint subfamily $$\{ I_{ \alpha_k} \}_{k \leq N} $$ such that $$ \bigcup_{\alpha \in A} I_\alpha \subset  \bigcup_{k \leq N} 3 I_{\alpha_k} $$
\end{lemma}

We apply Austin's covering lemma to prove two geometric estimates on intervals of the real line. The first one concerns union of dilated intervals.

\begin{lemma}\label{ L : interval dil }
Fix $\tau > 0$ and let $\{ I_\alpha \}_{\alpha \in A}$ a finite family of bounded intervals on $\mathbb{R}$. We have $$B_\tau \times \left|\bigcup_{\alpha \in A} \tau I_\alpha \right| \leq  \left|\bigcup_{\alpha \in A} I_\alpha \right| \leq C_\tau \times \left|\bigcup_{\alpha \in A} \tau I_\alpha \right|$$ where $C_\tau = \sup \{ \tau , \frac{1}{\tau} \}$ and $B_\tau = \inf \{ \tau , \frac{1}{\tau} \}$. In other words we have $$  \left|\bigcup_{\alpha \in A} I_\alpha \right| \simeq_{\tau} \left|\bigcup_{\alpha \in A} \tau I_\alpha \right|.$$
\end{lemma}

\begin{proof}

Suppose that $\tau > 1$. We just need to prove that $$ \left|\bigcup_{\alpha \in A} \tau I_\alpha \right| \leq \tau \left|\bigcup_{\alpha \in A} I_\alpha \right|.$$ Simply observe that we have $$ \bigcup_{\alpha \in A} \tau I_\alpha \subset \left\{ M \mathbb{1}_{ \cup_{ \alpha \in A} I_\alpha } > \frac{1}{\tau}\right\} $$ and apply the one dimensional maximal Theorem.

\end{proof}

Now that we have dealt with union of dilated intervals we consider union of translated intervals.

\begin{lemma}\label{L : interval}
Let $\mu> 0$ be a positive constant. For any finite family of intervals $\left\{I_\alpha \right\}_{\alpha \in A}$ on $\mathbb{R}$ and any finite family of scalars $\left\{t_\alpha \right\}_{\alpha \in A} \subset \mathbb{R}$ such that, for all $\alpha \in A$ $$ |t_\alpha| < \mu\times |I_\alpha|$$ we have $$  \left|\bigcup_{\alpha \in A} I_\alpha \right| \simeq_{\mu} \left|\bigcup_{\alpha \in A} \left( t_\alpha + I_\alpha \right) \right|.$$
\end{lemma}

\begin{proof}
We apply Austin's covering lemma to the family $\left\{ I_\alpha\right\}_{\alpha \in A}$ which gives a  disjoint subfamily $\left\{ I_{\alpha_k} \right\}_{ k \leq N}$ such that $$ \bigcup_{\alpha \in A} I_\alpha \subset \bigcup_{k \leq N} 3 I_{\alpha_k}.$$ In particular we have $$ \left| \bigsqcup_{k \leq N} I_{\alpha_k} \right| \simeq \left|\bigcup_{\alpha \in A} I_\alpha \right|.$$ We consider now the family $$  \left\{ (1+ \mu) I_{\alpha_k} \right\}_{ k \leq N} $$ which is \textit{a priori} not disjoint. We apply again  Austin's covering lemma which gives a  disjoint subfamily that we will denote $\left\{ (1+\mu) I_{\alpha_{k_l}} \right\}_{ l \leq M}$ who satisfies $$ \bigcup_{k \leq N}  (1+\mu) I_{\alpha_k} \subset \bigcup_{l \leq M} 3 (1+\mu) I_{\alpha_{k_l }}.$$ In particular we have $$ \left| \bigsqcup_{l \leq M}  (1+\mu) I_{\alpha_{k_l}} \right| \simeq \left|\bigcup_{k \leq N} (1+\mu) I_{\alpha_k} \right|.$$ To conclude, it suffices to observe that for any $\alpha \in A$ we have $$t_\alpha + I_\alpha \subset (1+\mu) I_\alpha $$ because $|t_\alpha| \leq \mu \times \left|I_\alpha\right|$. Hence the family $$ \{ t_{\alpha_{k_l}} + I_{\alpha_{k_l}}  \}_{l \leq M}$$ is disjoint and so finally $$\left| \bigsqcup_{l \leq M} \left( t_{\alpha_{k_l}} + I_{\alpha_{k_l}} \right) \right| = \sum_{l \leq M} \left|I_{\alpha_{k_l}} \right| \geq \frac{1}{3(1+\mu) } \left| \bigcup_{l \leq M} 3  (1+\mu)  I_{\alpha_{k_l}} \right| \simeq_{\mu} \left|\bigcup_{\alpha \in A} I_\alpha \right| $$ where we have used lemma \ref{ L : interval dil } in the last step.
\end{proof}

\subsection*{Geometric estimates on $\mathbb{R}^2$}

We denote by $\mathcal{S}$ the set containing \textit{all} parallelograms $u \subset \mathbb{R}^2$ whose vertices are of the form $(p,a),(p,b),(q, c)$ and $(q,d)$ where $p-q > 0$ and  $b-a = d-c > 0$. We say that $l_u := p-q$ is the \textit{length} of $u$ and that $w_u := b-a$ is the \textit{width} of $u$ ; we do not have necessarily $ l_u \geq w_u$. For $u \in \mathcal{S}$ and and a positive ratio $0 < \tau < 1$ we denote by $\mathcal{S}_{u,\tau}$ the collection defined as $$ \mathcal{S}_{u,\tau} := \left\{ s \in \mathcal{S} : s \subset u, l_s = l_u, |s| \geq \tau |u| \right\}.$$ We won't use directly the following proposition but its proof is instructive.

\begin{prp}[geometric estimate I]\label{L : First Geometric Observation}
Fix $\tau > 1$ and any finite family of parallelograms $\left\{ u_i \right\}_{ i \in I } \subset \mathcal{S}$. For each $i \in I$, select an element $s_i \in \mathcal{S}_{u_i, \tau}$. The following holds $$ \left| \bigcup_{i \in I} s_i \right| \geq \frac{\tau}{3}\left|\bigcup_{i \in I} u_i \right|.$$
\end{prp}

\begin{proof}
We let $U = \bigcup_{i \in I} u_i$ and $V = \bigcup_{i \in I} s_i $. Fix $x \in \mathbb{R}$ and for $i \in I$, denote by $u_i^x$ and $s_i^{x}$ the segments $u_i \cap \left\{ {x} \times \mathbb{R} \right\}$ and $s_i  \cap \left\{ {x} \times \mathbb{R} \right\}$. Observe that we have by hypothesis  $|s_i^{x}| \geq \tau |u_i^x|$. By definition, we have the following equality $$ \left|\bigcup_{i \in I} u_i^x\right| =  \int \mathbb{1}_U(x,y)dy$$ and as well as $$ \left|\bigcup_{i \in I} s_i^x\right| =  \int \mathbb{1}_V(x,y)dy.$$ We apply Austin's covering lemma to the family $\left\{ u_i^x \right\}_{i \in I}$ which gives a subfamily  $J \subset I$ such that the segments $\left\{ u_j^x \right\}_{j \in J} $ are disjoint intervals satisfying $$ \bigcup_{i \in I} u_i^x \subset \bigcup_{j \in J} 3 u_j^x.$$ This yields $$\left|\bigcup_{i \in I} s_i^{x}\right| \geq \sum_{j \in J} \left|s_j^x \right| \geq \frac{\tau}{3} \left|\bigcup_{i \in I} u_i^x \right|. $$ An integration over $x \in \mathbb{R}$ concludes the proof.
\end{proof}

We aim to give a more general version of proposition \ref{L : First Geometric Observation} using lemma \ref{ L : interval dil } and \ref{L : interval}. For $u \in \mathcal{S}$ define the parallelogram $h_u \in \mathcal{S}$ as \textit{the parallelogram who has same length, orientation and center than $u$ but is $5$ times wider} \textit{i.e.} $w_{h_u} = 5w_p$.

%%%%%%%%%%%%%%%%%%%%%
%%%%%%%%%%%%%%%%%%%%%
%%%%%%%%%%%%%%%%%%%%%
%%%%%%%%%%%%%%%%%%%%%
%%%%%%%%%%%%%%%%%%%%%
%%%%%%%%%%%%%%%%%%%%%
%%%%%%%%%%%%%%%%%%%%%
%%%%%%%%%%%%%%%%%%%%%
%%%%%%%%%%%%%%%%%%%%%
\begin{figure}[h!]
\centering
\includegraphics[scale=0.9]{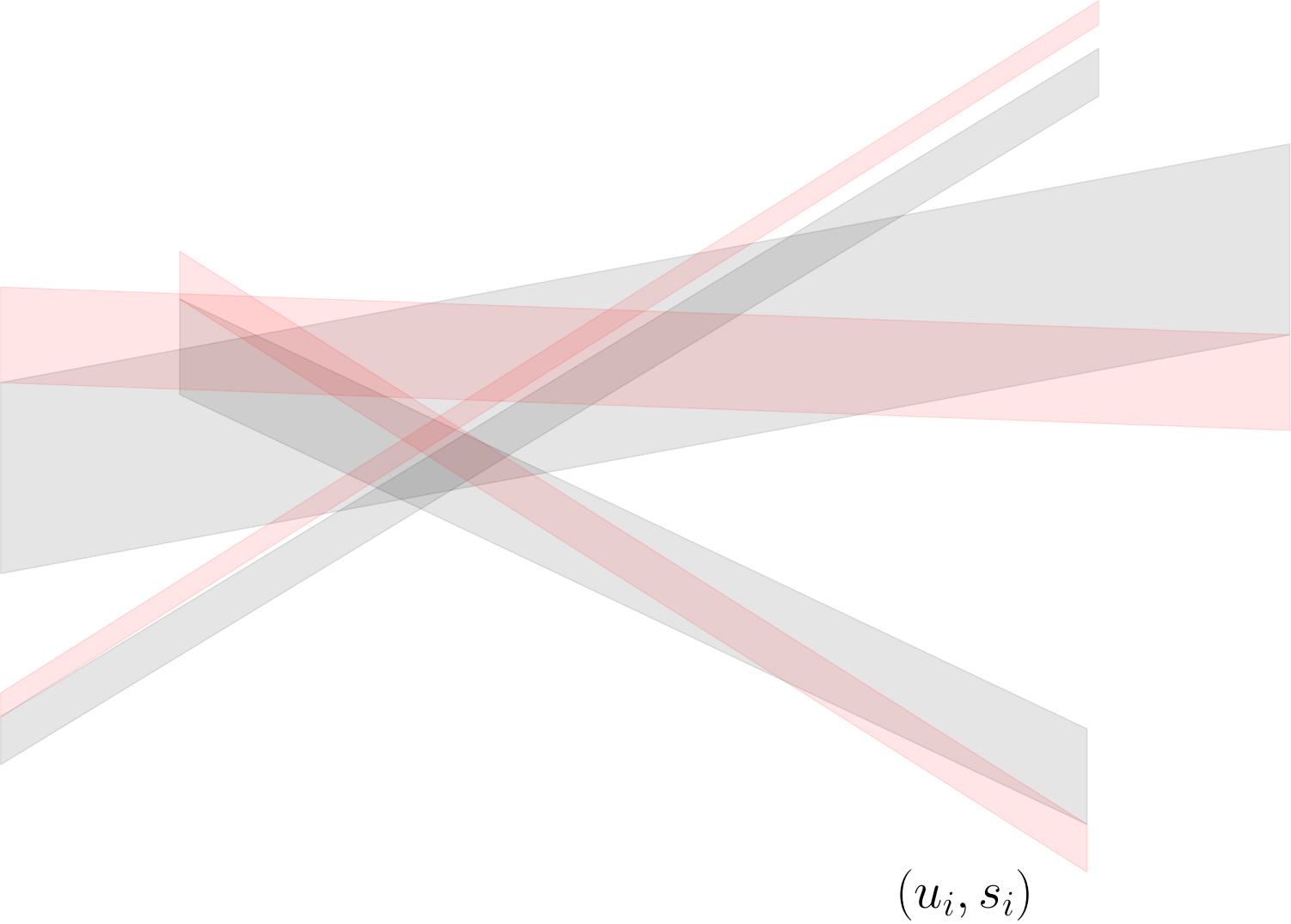}
\caption{An illustration of Proposition \ref{ prp : estimate geometric ; version décalé} with parameter $\tau = \frac{1}{4}$. The red area is bigger than $\simeq \tau$ times the grey area.}  
\end{figure}
%%%%%%%%%%%%%%%%%%%%%
%%%%%%%%%%%%%%%%%%%%%
%%%%%%%%%%%%%%%%%%%%%
%%%%%%%%%%%%%%%%%%%%%
%%%%%%%%%%%%%%%%%%%%%
%%%%%%%%%%%%%%%%%%%%%
%%%%%%%%%%%%%%%%%%%%%
%%%%%%%%%%%%%%%%%%%%%
%%%%%%%%%%%%%%%%%%%%%

\begin{prp}[geometric estimate II]\label{ prp : estimate geometric ; version décalé}
Fix $ 0 < \tau < 1$ and any finite family of parallelograms $\left\{u_i\right\}_{i \in I} \subset \mathcal{S}$. For each $i \in I$, select an element $s_i \in \mathcal{S}_{h_{u_i}, \tau}$. The following estimate holds $$\left|\bigcup_{i \in I} s_i \right| \geq \frac{\tau}{54}\left|\bigcup_{i \in I} u_i\right|.$$
\end{prp}

\begin{proof}
As in the proof of lemma \ref{L : First Geometric Observation}, denote $U = \bigcup_{i \in I} u_i$ and $V = \bigcup_{i \in I} s_i $. Fix $x \in \mathbb{R}$ and for $i \in I$, denote by $u_i^x$ and $s_i^{x}$ the segments $u_i \cap \left\{ {x} \times \mathbb{R} \right\}$ and $s_i  \cap \left\{ {x} \times \mathbb{R} \right\}$. For any $i \in I$, observe that there is a scalar $t_i$ satisfying $|t_i| \leq \mu \times |u_i|$ with $$ \mu = 5 $$ such that $$t_i + \tau u_i^x \subset s_i^x.$$ Applying lemma \ref{L : interval}, we then have (since $9 \times (1+\mu) = 54$) $$\left|\bigcup_{i \in I} s_i^x\right| \geq \left| \bigcup_{i \in I} \left(t_i + \tau u_i^x\right) \right| \geq \frac{1}{54} \left|\bigcup_{i \in I} \tau u_i^x\right|.$$ We conclude using lemma \ref{ L : interval dil } $$\frac{1}{54} \left|\bigcup_{i \in I} \tau u_i^x\right| \geq \frac{\tau}{54}  \left|\bigcup_{i \in I}  u_i^x\right|$$ by integrating on $x$ as before. 

\end{proof}

\subsection*{Geometric estimate involving a maximal operator}

We state a last geometric estimate involving maximal operator that will turn out to be crucial and we begin by a specific case. Consider $u := [0,1]^2$, $u' := \overrightarrow{(1,0)} + u$ and any element $v \in \mathcal{S}$ included in $u$ such that $l_v = l_u $ and $|v| \leq \frac{1}{2}|u|$.

\begin{prp}
There is a parallelogram $s \in \mathcal{S}_{h_u, \frac{1}{4}}$ depending on $v$ such that the following inclusion holds $$s \subset \left\{ M_v \mathbb{1}_{u'} > \frac{1}{16} \right\}.$$
\end{prp}

\begin{proof}

Without loss of generality, we can suppose that the lower left corner of $v$ is $O$. The upper left corner of $v$ is the point $(0,w_v)$ and we denote by $(d,1)$ and $(d+w_v,1)$ its lower right and upper right corners. Since $v \subset u$ we have $$d + w_v \leq 1.$$ The upper right corner of $\frac{1}{2}v$ is the point $(\frac{1}{2} (d + w_v), \frac{1}{2})$ and so for any $0 \leq y \leq 1 - \frac{1}{2}(d+w_v) $ we have $$(0,y) + \frac{1}{2}v \subset u.$$ This yields our inclusion as follow. Let $\Vec{t} \in \mathbb{R}^2$ be a vector such that the center of the parallelogram $\Tilde{v} = \Vec{t} + 2v$ is the point $(1,0)$. By construction we directly have $$| \Tilde{v} \cap u'| \geq \frac{1}{16}$$ but moreover for any $0 \leq y \leq \frac{1}{2}$ we have $$|  \left\{ (0,y) + \Tilde{v} \right\}  \cap c'| \geq \frac{1}{16}$$ since the upper right quarter of $\Tilde{v}$ is relatively to $u'$ in the same position than $v$ relatively to $u$. Finally, denoting by $v^*$ the parallelogram $\Tilde{v} \cap [0,1] \times \mathbb{R} $, the parallelogram $s$ defined as $$s := \bigcup_{0 \leq y \leq  \frac{1}{2}} \left( (0,y) + v^* \right)  $$ satisfies the condition claimed. This concludes the proof.
\end{proof}

We state now the previous proposition in its general form. We fix an arbitrary element $u \in \mathcal{P}$ and an element $v \in \mathcal{S}$ included in $u$ such that $l_v = l_u$ and $|v| \leq \frac{1}{2}|u|$. Recall that we denote by $u'$ the parallelogram $u$ translated of one unit length in its direction.
  
%%%%%%%%%%%%%%%%%%%%%
%%%%%%%%%%%%%%%%%%%%%
%%%%%%%%%%%%%%%%%%%%%
%%%%%%%%%%%%%%%%%%%%%
%%%%%%%%%%%%%%%%%%%%%
%%%%%%%%%%%%%%%%%%%%%
%%%%%%%%%%%%%%%%%%%%%
%%%%%%%%%%%%%%%%%%%%%
%%%%%%%%%%%%%%%%%%%%%
\begin{figure}[h!]
\centering
\includegraphics[scale=0.9]{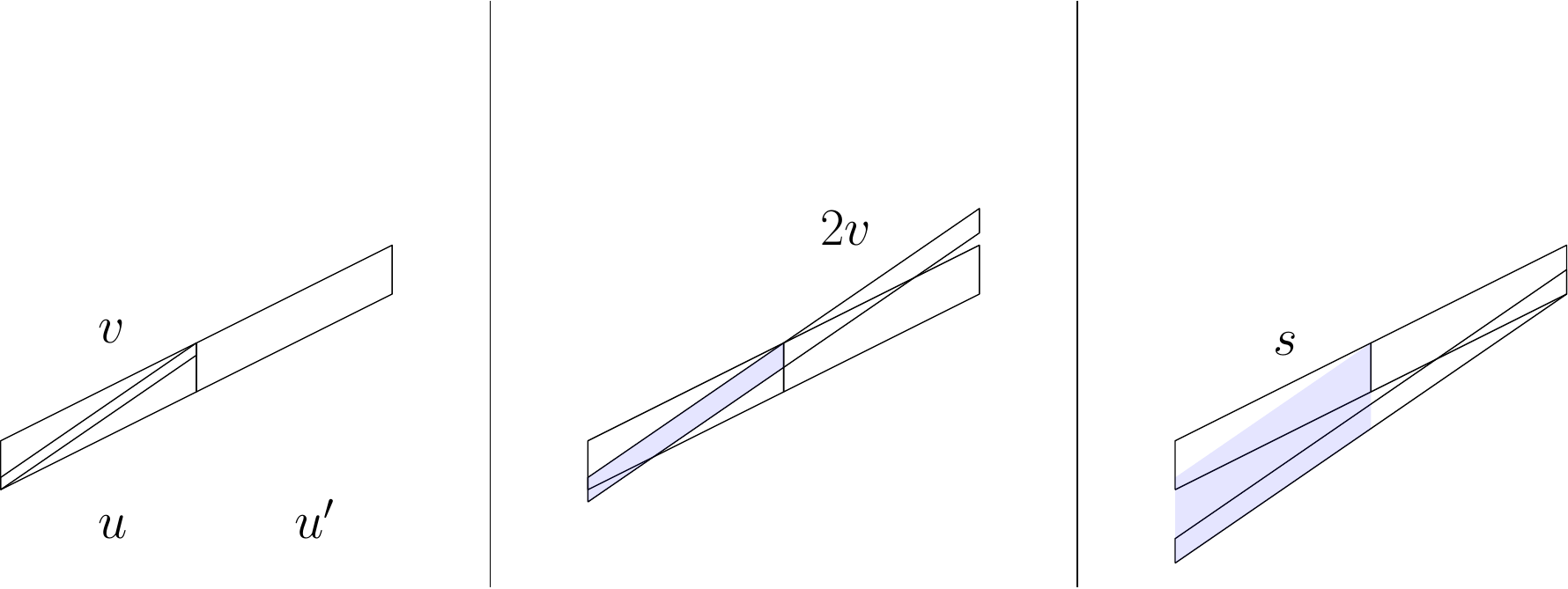}
\caption{An illustration of Proposition \ref{L : final lemma ge}.}  
\end{figure}
%%%%%%%%%%%%%%%%%%%%%
%%%%%%%%%%%%%%%%%%%%%
%%%%%%%%%%%%%%%%%%%%%
%%%%%%%%%%%%%%%%%%%%%
%%%%%%%%%%%%%%%%%%%%%
%%%%%%%%%%%%%%%%%%%%%
%%%%%%%%%%%%%%%%%%%%%
%%%%%%%%%%%%%%%%%%%%%
%%%%%%%%%%%%%%%%%%%%%
 
\begin{prp}\label{L : final lemma ge}
There is parallelogram $s \in \mathcal{S}_{ h_u,\frac{1}{4}}$ depending on $v$ such that the following inclusion holds $$  s \subset \left\{ M_v \mathbb{1}_{u'} > \frac{1}{4} \right\}.$$
\end{prp}

\begin{proof}
There is a unique linear function $f : \mathbb{R}^2 \rightarrow \mathbb{R}^2$ with positive determinant such that $f(u) = [0,1]^2$. Using this function and the previous lemma, the conclusion comes.
\end{proof}

\section{Proof of Theorem \ref{ T : main theorem} }\label{ S : proof thm main }

We fix an arbitrary family $\mathcal{B}$ contained in $\mathcal{T}$ and $1 < p < \infty$. We are going to prove that one has $$ B_p \times \log( \lambda_{[\mathcal{B}]} ) \leq \| M_{\mathcal{B}} \|_p^p.$$ To do so, we will prove that $M_\mathcal{B}$ admits a Kakeya-type set of level $$\left(\frac{1}{16}, C \log(n)^{-1}\right)\simeq \left(\frac{1}{2}, \log(n)^{-1}\right).$$

\subsection*{Strategy}

The family $\mathcal{B}$ generates a tree $[\mathcal{B}]$ ; we fix a fig tree $[\mathcal{F}] \subset [\mathcal{B}]$ of scale $\lambda_{[\mathcal{B}]}$ and we denote by $h \in \mathbb{N}$ its height.  Consider as before the random set $A$ associated to $[\mathcal{F}]$  $$A := \bigcup_{k \leq 2^h} \Vec{t}_k + (r_k \cup r_k').$$ We fix a realisation $\omega \in \Omega$ such that $|A_2(\omega)| \geq \frac{\log(n)}{C}|A_1(\omega)|$. We take advantage of $A_1:=A_1(\omega)$ but this time using  elements of $\mathcal{B}$ and not elements of $[\mathcal{F}]$.

%%%%%%%%%%%%%%%%%%%%%
%%%%%%%%%%%%%%%%%%%%%
%%%%%%%%%%%%%%%%%%%%%
%%%%%%%%%%%%%%%%%%%%%
%%%%%%%%%%%%%%%%%%%%%
%%%%%%%%%%%%%%%%%%%%%
%%%%%%%%%%%%%%%%%%%%%
%%%%%%%%%%%%%%%%%%%%%
%%%%%%%%%%%%%%%%%%%%%
\begin{figure}[h!]
\centering
\includegraphics[scale=0.9]{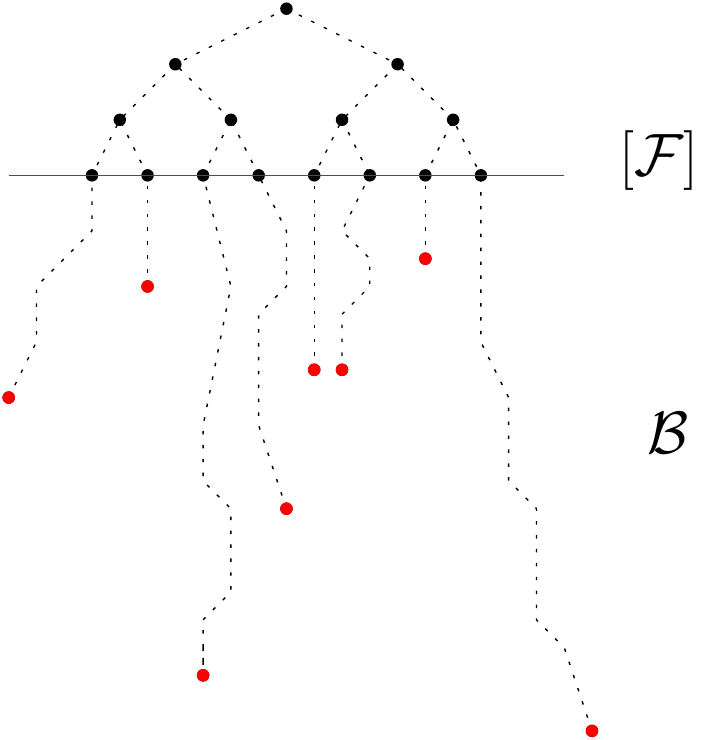}
\caption{Theorem \ref{ T : main theorem} shows that we can virtually use the tree $[\mathcal{F}]$ for the operator $M_\mathcal{B}$ even if $\mathcal{B}$ has no structure. On the illustration, $\mathcal{B}$ is composed of the red dots which represent rectangles who have very different scale and yet they \textit{interact} at the level of $[\mathcal{F}]$.}  
\end{figure}
%%%%%%%%%%%%%%%%%%%%%
%%%%%%%%%%%%%%%%%%%%%
%%%%%%%%%%%%%%%%%%%%%
%%%%%%%%%%%%%%%%%%%%%
%%%%%%%%%%%%%%%%%%%%%
%%%%%%%%%%%%%%%%%%%%%
%%%%%%%%%%%%%%%%%%%%%
%%%%%%%%%%%%%%%%%%%%%
%%%%%%%%%%%%%%%%%%%%%

\subsection*{Applying Proposition \ref{L : final lemma ge}}

For any $u \in \mathcal{L}_{[\mathcal{F}]}$ we fix an element $v_u$ of $\mathcal{B}$ such that $v_u \subset u$. To each pair $(u,v_u)$ we apply Proposition \ref{L : final lemma ge} and this gives a parallelogram $s_u \in \mathcal{S}_{h_u,\frac{1}{4}}$ such that $s_u \subset \left\{ M_{v_u}\mathbb{1}_{u'} > \frac{1}{16} \right\}$. We define then the set $B_2$ as $$B_2 := \bigcup_{ k \leq 2^h}  \Vec{t}_k + s_{r_k}' $$ Because $v_u \in \mathcal{B}$ we obviously have $M_{v_u} \leq M_\mathcal{B}$ and so  $s_u \subset \left\{ M_{\mathcal{B}}\mathbb{1}_{u'} > \frac{1}{16} \right\}$. Considering the union over $k \leq 2^h$ we obtain $$B_2 := \bigcup_{k \leq 2^h} \Vec{t}_k + s_{r_k}' \subset \left\{ M_{\mathcal{B}}\mathbb{1}_{A_1} > \frac{1}{16} \right\}$$ and so finally $|B_2| \leq \left|\left\{ M_{\mathcal{B}}\mathbb{1}_{A_1} > \frac{1}{16} \right\}\right|$.

\subsection*{Applying Proposition \ref{ prp : estimate geometric ; version décalé}}

It remains to compute $|B_2|$ ; to do so we observe that we can use proposition \ref{ prp : estimate geometric ; version décalé} with the families $\left\{ \Vec{t}_k + r_k' \right\}_{k \leq 2^h}$ and $\left\{  \Vec{t}_k + s_{r_k}' \right\}_{k \leq 2^h}$. This yields $$|B_2| \geq \frac{1}{21 \times 4}|A_2|$$ and so we finally have $$|A_1| \lesssim \frac{1}{\log(n)}\left| \left\{ M_{\mathcal{B}}\mathbb{1}_{A_1} > \frac{1}{16} \right\}\right|.$$ In other words, the set $A_1$ is a Kakeya-type set of level $\simeq (\frac{1}{2}, \log(n)^{-1})$ for the maximal operator $M_\mathcal{B}$ and this concludes the proof of Theorem \ref{ T : main theorem}.

\section{Proof of Theorem \ref{T:app1}}

Let $\Omega$ be a bad set of directions in $[0,\frac{\pi}{4})$ and let $\mathcal{B}$ be a rarefied basis of $\mathcal{R}_\Omega$ \textit{i.e.} we have $\mathcal{B} \subset \mathcal{R}_\Omega$ and also $$\sup_{ \omega \in \Omega} \inf_{ r \in \mathcal{B}, \omega_r = \omega} e_r  = 0.$$ Let's denote $\mathcal{T}_\Omega$ be the family associated to $\mathcal{R}_\Omega$ by Proposition \ref{ P : approx } ; observe now that our hypothesis implies that we have $[\mathcal{B}] = \mathcal{T}_\Omega$ and so in particular we have $$ \lambda_{[\mathcal{B}]} = \lambda_{ \mathcal{T}_\Omega}.$$ The following claim will concludes the proof.

\begin{claim}
If $\Omega$ is a bad set of directions then $\lambda_{ \mathcal{T}_\Omega} = \infty$.
\end{claim}

Applying Theorem \ref{ T : main theorem}, we obtain for any $1 < p < \infty$ $$ \infty = \lambda_{\mathcal{T}_\Omega} = \lambda_{[ \mathcal{B}]} \lesssim \|M_\mathcal{B}\|_p^p.$$

\section{Proof of Theorem \ref{T :app2}}

For $n \in \mathbb{N}^*$ recall that $r_n \in \mathcal{R}$ is a rectangle satisfying $$\left( e_{r_n}, \omega_{r_n} \right) = \left( \frac{1}{n} , \sin(n) \frac{\pi}{4} \right).$$ We let $\mathcal{B}$ be the basis generated by the rectangles $\{r_n \}_{ n \geq 1}$ ; we are going to prove that $$ \lambda_{[\mathcal{B}]} = \infty$$ and then apply Theorem \ref{ T : main theorem} to prove Theorem \ref{T :app2}. To begin with, observe that we have $$\lim_{n \infty} e_{r_n} = 0 $$ and also that $$\textbf{Adh}\left( \{ \omega_{r_n} , n \in \mathbb{N}^* \} \right) = [0,\frac{\pi}{4}]$$ where $\textbf{Adh}(E)$ denotes the topological adherence of the set $E$. Fix a large integer $H \gg 1$ and for $0 \leq k \leq 2^H-1$ let $$I_k := [ \frac{k}{2^H} \frac{\pi}{4}, \frac{k+1}{2^H} \frac{\pi}{4}].$$

\begin{claim}
For any $H \gg 1$ and any $ k \leq 2^H-1$ there exists $n \gg 1$ such that we have $\omega_{r_n} \in I_k$ and $e_{r_n} \leq \frac{1}{2^H}$.
\end{claim}

It easily follows by the claim that we have $\lambda_{[\mathcal{B}]} = \infty$ which concludes the proof of Theorem \ref{T :app2}.

{}


\begin{thebibliography}{}

\end{thebibliography}


\begin{thebibliography}{}


\bibitem{ALFONSECA}
M. A. Alfonseca, \emph{Strong type inequalities and an almost-orthogonality principle for families of maximal operators along directions in $\mathbb{R}^2$ } J. London Math. Soc. 67 no. 2 (2003), 208-218.

\bibitem{AUSTIN}
D. Austin, \emph{A geometric proof of the Lebesgue differentiation theorem}, Proceedings of the American Mathematical Society 16: (1965) 220–221.




\bibitem{BATEMAN}
M. D. Bateman, \emph{Kakeya sets and directional maximal operators in the plane}, Duke Math. J.
147:1, (2009), 55–77.

\bibitem{BATEMANKATZ}
M. D. Bateman and N.H. Katz, \emph{Kakeya sets in Cantor directions}, Math. Res. Lett. 15 (2008), 73–81.

\bibitem{CORDOBAFEFFERMAN}
A. Corboda and R. Fefferman, \emph{A Geometric Proof of the Strong Maximal Theorem} Annals of Mathematics. vol. 102, no. 1, 1975, pp. 95–100.

\bibitem{CORDOBAFEFFERMAN II}
A. Cordoba and R. Fefferman, \emph{On differentiation of integrals}, Proc. Nat. Acad. Sci.
U.S.A. 74:6, (1977), 2211–2213.




\bibitem{GAUVAN}
A. Gauvan
\emph{Application of Perron Trees to Geometric Maximal Operators}, HAL, https://hal.archives-ouvertes.fr/hal-03295909.




\bibitem{KATHRYN JAN}
K. Hare and J.-O. Rönning, \emph{Applications of generalized Perron trees to maximal functions and density bases}, J. Fourier Anal. and App. 4 (1998), 215–227.

\bibitem{NSW}
A. Nagel, E. M. Stein, and S. Wainger, \emph{ Differentiation in lacunary directions}, Proc.
Nat. Acad. Sci. U.S.A. 75:3, (1978), 1060–1062.


\end{thebibliography}
\end{document}